\documentclass[12pt]{amsart}
\usepackage[utf8]{inputenc}

\usepackage{tikz-cd}
\usepackage{verbatim} %
\usepackage{graphicx}
\usepackage{mathtools}
\usepackage{mathrsfs}
\usepackage{enumerate}
\usepackage[T2A]{fontenc}
\usepackage[utf8]{inputenc}
\usepackage{MnSymbol}

\usepackage{tikz-cd}

\DeclareMathAlphabet\mathbb{U}{msb}{m}{n}

\def\ZZ{\mathbb{Z}}

\numberwithin{equation}{section}

\newtheorem{Theorem}{Theorem}[section]
\newtheorem{Corollary}[Theorem]{Corollary}
\newtheorem{Lemma}[Theorem]{Lemma}
\newtheorem{Proposition}[Theorem]{Proposition}
\theoremstyle{definition}
\newtheorem{Definition}[Theorem]{Definition}
\newtheorem{Remark}[Theorem]{Remark}

\def\epi{\twoheadrightarrow}
\def\mono{\rightarrowtail}
\def\CC{\mathcal C}
\def\DD{\mathcal D}
\def\BB{{M}}

\begin{document}

\title[Right exact group completion]{ Right exact group completion as a transfinite invariant of homology equivalence}

\author{Sergei O. Ivanov}
\address{
Laboratory of Modern Algebra and Applications,  St. Petersburg State University, 14th Line, 29b,
Saint Petersburg, 199178 Russia}\email{ivanov.s.o.1986@gmail.com}

\author{Roman Mikhailov}
\address{Laboratory of Modern Algebra and Applications, St. Petersburg State University, 14th Line, 29b,
Saint Petersburg, 199178 Russia and St. Petersburg Department of
Steklov Mathematical Institute} \email{rmikhailov@mail.ru}

\thanks{The main result of the paper (Theorem 4.1) was obtained under the support
of the Russian Science Foundation grant N 16-11-10073. The authors also are
supported by the grant of the Government of the Russian Federation for the
state support of scientific research carried out under the supervision of leading
scientists, agreement 14.W03.31.0030 dated 15.02.2018.}

\maketitle

\begin{abstract}
We consider a functor from the category of groups to itself $G\mapsto \mathbb Z_\infty G$ that we call right exact $\mathbb Z$-completion of a group. It is connected with the pronilpotent completion $\widehat{G}$ by the short exact sequence $1\to {\varprojlim}^1\: M_n G \to \mathbb Z_\infty G \to \widehat G \to 1,$ where $M_n G$ is $n$-th Baer invariant of $G.$ We prove that $\mathbb Z_\infty (\pi_1X)$ is an invariant of homological equivalence of a space $X$. Moreover, we prove an analogue of Stallings' theorem: if $G\to G'$ is a 2-connected group homomorphism, then $\mathbb Z_\infty G\cong \mathbb Z_\infty G'.$ We give examples of $3$-manifolds $X,Y$ such that $ \widehat{\pi_1 X}\cong \widehat{\pi_1 Y}$ but $\mathbb Z_\infty \pi_1X\not \cong \mathbb Z_\infty \pi_1Y.$ We prove that for a group $G$ with finitely generated $H_1G$ we have $(\mathbb Z_\infty G)/ \gamma_\omega= \hat G.$ So the difference between $\hat G$ and $\mathbb Z_\infty G$ lies in $\gamma_\omega.$ This allows us to treat $\mathbb Z_\infty \pi_1X$ as a transfinite invariant of $X.$ The advantage of our approach is that it can be used not only for $3$-manifolds but for arbitrary spaces.
\end{abstract}

\section*{\bf Introduction}

For a group $G$ we denote by $\gamma_n G$ its lower central series and by $\widehat{G}$ the pronilpotent completion
$$\widehat G=\varprojlim\ G/\gamma_n G.
$$
It is well-known that the pronilpotent completion of the fundamental group is an invariant of $\ZZ$-homological equivalence, i.e. if $X\to Y$ is a $\ZZ$-homological equivalence of spaces, then $ \widehat{\pi_1X} \to \widehat{ \pi_1Y}$ is an isomorphism (see \cite[Th.5.1]{Stallings}). In this paper we present a  functor from the category of groups to itself $G\mapsto \ZZ_\infty G,$ that we call the {\it right exact $\ZZ$-completion}, which gives a stronger invariant of homological equivalence than the pronilpotent completion. Moreover, we give examples of $3$-manifolds $X,Y$, such that $\widehat{\pi_1 X}\cong  \widehat{\pi_1 Y}$ but $\ZZ_\infty \pi_1X \not\cong \ZZ_\infty \pi_1 Y.$

If $X$ is a space we denote by $\ZZ_\infty X$ its Bousfield-Kan $\ZZ$-completion \cite{BousfieldKan}. By definition we set
$$\ZZ_\infty G:=\pi_1(\ZZ_\infty BG),$$
where $BG$ is the classifying space. The homotopy type of the Bousfield-Kan $\ZZ$-completion $\ZZ_\infty X$ is invariant with respect to $\ZZ$-homology equivalence i.e. if $X\to Y$ is a $\ZZ$-homology equivalence, then $\ZZ_\infty X\to \ZZ_\infty Y$ is a homotopy equivalence  \cite{BousfieldKan}. Therefore $\pi_1(\ZZ_\infty X )$ is invariant with respect to $\ZZ$-homological equivalence. We prove the following (see Proposition \ref{prop_Z-inf-X})

\

\noindent {\bf Proposition.} {\it Let $X$ be a path-connected pointed space. Then there is a natural isomorphism
$$ \pi_1(\ZZ_\infty X )\cong \ZZ_\infty (\pi_1X).$$}

\noindent In particular, we obtain  that $\ZZ_\infty \pi_1(X)$ is an invariant of $\ZZ$-homological equivalence.

For a group $G$ we denote by $\BB_n G$ its $n$-th Baer invariant (see \cite{BurnsEllis97}, \cite{MacDonald},  \cite{Ellis} or Section \ref{sec_baer} for the definition of Baer invariants).
 We prove that there is a short exact sequence which describes  the difference between $\ZZ_\infty G$ and $\widehat G:$
$$1 \longrightarrow {\varprojlim}^1\: \BB_n G \longrightarrow \ZZ_\infty G \longrightarrow  \widehat G \longrightarrow 1.$$
  If the first homology of the group is finitely generated, we have even better interrelation between $\ZZ_\infty G$ and $\widehat G$ (see Proposition \ref{prop_gamma_omega}):

\

\noindent {\bf Proposition.} {\it If $G$ is a group with finitely generated $H_1G,$ then there are natural isomorphisms
$$\gamma_\omega(\ZZ_\infty G)\cong {\varprojlim}^1\BB_n G, \hspace{1cm} \ZZ_\infty G/\gamma_\omega (\ZZ_\infty G) \cong \widehat G. $$}

\noindent This proposition implies that for groups with finitely generated $H_1G$ the group $\widehat G$ can be reconstructed from the group $\ZZ_\infty G.$ In other words, the right exact $\ZZ$-completion carries more information about $G$ than the usual pronilpotent completion.

The main tool for computations of $\ZZ_\infty G$ is the following (see Corollary \ref{prop_resoluion}):

\

\noindent {\bf Proposition.} {\it If $U\mono E \epi G$ is a short exact sequence of groups such that $H_1E$ is finitely generated and $H_2 E$ is finite, then there is an exact sequence
$$\widehat U_E \longrightarrow \widehat E \longrightarrow \ZZ_\infty G \longrightarrow 1.$$}

We construct examples of $3$-manifolds that can't be distinguished by $ \widehat\pi_1$ but can be distinguished by $\ZZ_\infty \pi_1.$ Consider the following matrix $a_k= \left( \begin{smallmatrix} -1 & k \\
0 & -1 \end{smallmatrix} \right),$ where $k$ is an {\it odd} integer. The matrix defines the following semidirect product
$$G_k=\ZZ\ltimes_{a_k} \ZZ^2,$$
where $\ZZ$ acts on $\ZZ^2$ by $a_k.$  The construction of mapping torus  associated with the homeomorphism on $(S^1)^2$ induced by $a_k$ gives an aspherical $3$-manifold $X_k=K(G_k,1)$ with a fiber sequence
$$(S^1)^2 \to X_k \epi S^1$$
and the monodromy action induced by $a_k.$ The following result (see Theorem \ref{th_manifolds}) illustrates the value of the introduced invariants. 

\

\noindent{\bf Theorem.}\ {\it  Let $k$ and $l$ be two odd integers. Then the following holds.
\begin{enumerate}
\item $
\widehat{ \pi_1(X_k)} \cong \widehat{ \pi_1(X_l)}.$
\item If $kl \not\equiv 1,7\ ({\rm mod}\ 8),$ then $  \ZZ_\infty  \pi_1(X_k)\not \cong  \ZZ_\infty \pi_1(X_l).$
\end{enumerate}}

\

\noindent The condition $kl \not \equiv 1,7 \: ({\rm mod}\ 8)$ is connected with the fact that a $2$-adic integer $\alpha\in \ZZ_2$ is a square if and only if $\alpha \equiv 1\: ({\rm mod}\ 8).$ We use the group $\ZZ_\infty \pi_1$ in order to distinguish $3$-manifolds. However, the advantage of this invariant is that it can be used for arbitrary spaces.

In this paper we use the notion of right exact functor in the sense of Keune (see \cite{Keune}, \cite{AkhtiamovIvanovPavutnitskiy} and Section \ref{sec_right_exact}) and prove that, in some sense, $G\mapsto \ZZ_\infty G$ is the ``right exact version'' of the usual functor of pronilpotent completion $G\mapsto \widehat G$. That is why we call $\ZZ_\infty G$ the {\it right exact} $\ZZ$-completion of $G$.

The main motivation of this research is the old problem of J. Milnor, of finding a realizable transfinite version of $\mu$-invariants for links \cite{Milnor}. The pronilpotent completion of the link group provides a natural concordance invariant. How can one find a transfinite analog of this invariant, which can distinguish links with the same group completion up to concordance? We are not able to construct such an invariant for links, but one can extend the problem to the class of all spaces (or 3-manifolds) and change concordance by homology equivalence (or homology cobordism). The right exact $\ZZ$-completion of the fundamental group gives an answer on this extended problem.

Recall that, the original question asked by Milnor \cite{Milnor}, was about extracting invariants of links
from the quotient of the link group by the intersection of lower central series, i.e. by $\gamma_\omega$. In general,
the group $\mathbb Z_\infty \pi_1(X)$ is not an invariant of $\pi_1(X)/\gamma_\omega(\pi_1(X))$. However, if we restrict the class of
spaces $X$ to torus bundles, it will be. The reason is that, for a finitely generated metabelian group,
the lower central series stabilize on $\leq \omega$ step, in particular, the natural map $H_2(-)\to H_2(-/\gamma_\omega)$ is
an epimorphism. This implies that, for a torus bundle $X$, $\mathbb Z_\infty \pi_1(X)\cong \mathbb Z_\infty(\pi_1(X)/\gamma_\omega(\pi_1(X)))$. In
particular, $\mathbb Z_\infty \pi_1$ provides a solution of the original Milnor problem for the class of torus bundles.

Also we have to mention that the paper \cite{ChaOrr2013} played a significant role in developing transfinite methods in homology cobordism of 3-manifolds. In a forthcoming paper \cite{ChaOrrPrep}, Jae Choon Cha and Kent E. Orr have developed a full theory of Milnor invariants for 3-manifolds. Their homology cobordism invariants, indexed on finite and infinite ordinals, include Milnor’s classical link invariants as a special case, arising as invariants of 0-surgery on a link. As one corollary of their theory, they classify all possible transfinite lower central series quotients of the Vogel localized groups of oriented, closed 3-manifolds. Some of their examples  use the same twisted torus bundles. 

The authors thank the referee for the very detailed report and valuable comments about simplification of proofs.

\section{\bf Right exact functors in the sense of Keune}\label{sec_right_exact}

In this section we present a result of \cite{AkhtiamovIvanovPavutnitskiy} about right exact functors on the category of groups. Namely, we give several equivalent descriptions of right exact functors on the category of groups.

Let $\mathcal C$ be a category. A couple of epimorphisms $\alpha_0,\alpha_1: c' \rightrightarrows c $ is called {\it split couple of epimorphisms} if there exists a map $s:c'\to c$ such that $\alpha_0s ={\sf id}_c = \alpha_1 s.$ In other words, $s$ is a splitting for both epimorphisms $\alpha_0,\alpha_1$ at the same time. A functor
$$\Phi: \mathcal C \longrightarrow \mathcal D $$
is called {\it right exact} (in the  sense of Keune) if, for any split couple of morphisms $\alpha_0,\alpha_1:c'\rightrightarrows  c,$ the couple $(\alpha_0,\alpha_1)$ has a coequalizer in $\CC,$ the couple $(\Phi \alpha_0,\Phi \alpha_1) $ has a coequalizer in $\DD$ and the natural morphism
$${\rm coeq}(\Phi \alpha_0, \Phi \alpha_1) \to  \Phi({\rm coeq} (\alpha_0,\alpha_1 ))$$
is an isomorphism.

\begin{Remark} Some authors call a functor right exact if it commutes with finite colimits. We do not use this notion. By a right exact functor we always mean a right exact functor in this weaker sense of Keune.
\end{Remark}

Following Kan \cite{Kan}, \cite{DwyerKan}, we say that a simplicial group $F_\bullet$ is free if all groups $F_n$ are free and it is possible to chose bases of these groups so that they are stable under degeneracy maps. A free simplicial resolution is a weak equivalence of simplicial groups $F_\bullet \to G ,$ where $G$ is considered as a constant simplicial group and $F_\bullet$ is a free simplicial group. The derived functor of a functor $\Phi:{\sf Gr}\to {\sf Gr}$ is defined as follows ${\sf L}_n\Phi(G):=\pi_n(\Phi(F_\bullet)),$ where $F_\bullet \overset{\sim}\epi G$ is a free simplicial  resolution. Note that there is a natural map from the zero derived functor to the functor itself
$${\sf L}_0\Phi\to \Phi.$$

Let $G$ be a group. We say that $U$ is a $G$-group if $U$ is a group together with  a right action of $G$ on $U$ by automorphisms. In this case we can consider the semidirect product $G\ltimes U.$ Morphisms of $G$-groups are homomorphisms preserving the action of $G.$ The category of $G$-groups is denoted by $G$-${\sf Gr}$.
A normal subgroup $U$ of $G$ will be always considered as a $G$-group with the conjugation action.

Let $\Phi:{\sf Gr}\to {\sf Gr}$ be a functor. For a group $G$ we consider a functor from the category of $G$-groups to the category of $\Phi G$-groups
$$ \Phi_G :G\text{-}{\sf Gr} \longrightarrow \Phi G\text{-}{\sf Gr} $$
given by
$$\Phi_GU:={\sf Ker}(\Phi(G\ltimes U) \to \Phi G).$$
The action of $\Phi G$ on $\Phi_G U$ goes via the map $\Phi G\to \Phi(G\ltimes U).$ Note that there is an isomorphism
$$\Phi(G\ltimes U)=\Phi G\ltimes \Phi_GU.$$

Akhtiamov, Ivanov and Pavutnitskiy proved the following statement.

\begin{Theorem}[{\cite[Th.1.5]{AkhtiamovIvanovPavutnitskiy}}] \label{th_right_exact} The following statements about a functor $\Phi:{\sf Gr}\to {\sf Gr}$ are equivalent.

\begin{enumerate}
\item $\Phi$ is right exact.
\item The natural map ${\sf L}_0\Phi \to \Phi$ is an isomorphism.
\item  For a short exact sequence $U\mono G \epi H$ the sequence
$$ \Phi_G U \longrightarrow \Phi G \longrightarrow \Phi H \longrightarrow 1 $$
is exact, where the map $\Phi_G U \to \Phi G $ is induced by the map $G\ltimes U\to G, (g,u)\mapsto gu.$

\item For a simplicial group $G_\bullet$ the natural map $\pi_0( \Phi G_\bullet) \to \Phi ( \pi_0(G_\bullet))$ is an isomorphism.
\end{enumerate}
\end{Theorem}
\begin{Remark}\label{rem_l0} Note that for any functor $\Phi$ the functor ${\sf L}_0\Phi$ is right exact. Indeed, the functors $\Phi$ and ${\sf L}_0\Phi$ are equal on free groups, and hence ${\sf L}_0({\sf L_0}\Phi)\cong {\sf L}_0\Phi.$
\end{Remark}

\section{\bf Baer invariants of groups}\label{sec_baer}

It is known that integral homology of a group $H_*G=H_*(G,\ZZ)$ are derived functors of the functor of abelianization \cite[Ch. II. \S 5]{Quillen}. More precisely,  if we take a free simplicial resolution $F_\bullet \overset{\sim}\epi G$ then
$$H_{n+1}G=\pi_n((F_\bullet)_{\sf ab}).$$
There is another sequence of functors starting from $H_2G$ which can be defined in the language of derived functors.

\begin{Definition}[relative lower central series] For a group $G$ and its normal subgroup $U$ we define the $n$-th term of its lower central series $\gamma_n(U,G)$ by recursion
$$\gamma_{n+1}(U,G)=[\gamma_n(U,G),G],$$
where $\gamma_1(U,G)=U.$ Then the usual lower central series is defined as $\gamma_n G=\gamma_n(G,G).$ Moreover, we denote by  $\nu_n G$  the  quotient  $$\nu_n G=G/\gamma_{n+1}G.$$ Then $\nu_n G$ is called the $n$-nilpotenization of $G.$
\end{Definition}

\begin{Definition}[Baer invariant]\label{def_Baer}
  We define the Baer invariant $\BB_n$ as the first derived functor of the functor of $n$-nilpotenization:
$$\BB_n G=\pi_1(\nu_n F_\bullet).$$
Derived functors do not depend of the choice of the resolution, so $\BB_n G$ is an invariant of  $G.$ The first homotopy group of a simplicial group is abelian (because any simplicial group is weak equivalent to the loop space of its classifying space \cite[Ch.V]{GoerssJardine}). Hence $\BB_nG$ is an abelian group. Note that $\BB_1 G\cong H_2 G.$ If we present the group $G$ as a quotient of a free group $G=F/R,$ then there is the following version of Hopf's formula for the Baer invariant
$$\BB_n G =\frac{R \cap \gamma_{n+1}F }{\gamma_{n+1}(R,F)}, $$
where $\gamma_{n+1}(R,F)=[ \gamma_n(R,F),F ]$ and $\gamma_1(R,F)=R$ (see \cite{BurnsEllis97}, \cite{MacDonald},  \cite{Ellis}). Baer invariant $M_nG$ is also known as $n$-nilpotent multiplier which is the reason for the notation. Ellis proved the following theorem about Baer invariants.
\end{Definition}

\begin{Theorem}[{\cite[Th. 2]{Ellis}, \cite[Th. 1.74]{MikhailovPassiBook}}]\label{th_Ellis}
If $H_2G$ is a torsion group, then $\BB_nG$ is a torsion group for any $n.$
\end{Theorem}

The proof of the following proposition use the same strategy as the proof Ellis \cite[Th. 2]{Ellis}.

\begin{Proposition}\label{prop_f.p._baer} If $G$ is a group such that $H_1G$ and $H_2G$ are finitely generated, then $\BB_nG$ is finitely generated for any $n$.
\end{Proposition}
\begin{proof}
Let $F_\bullet\overset{\sim}\epi G$ be a free simplicial resolution. The short exact sequence $(\gamma_n F_\bullet)/(\gamma_{n+1} F_\bullet) \mono \nu_n F_\bullet\epi \nu_{n-1} F_\bullet$ gives an exact sequence
$$\pi_1( \gamma_n F_\bullet/\gamma_{n+1} F_\bullet) \longrightarrow \pi_1( \nu_n F_\bullet ) \longrightarrow \pi_1( \nu_{n-1} F_\bullet ).$$ So, in order to prove the statement by induction, it is enough to prove that $\pi_1( (\gamma_n F_\bullet)/(\gamma_{n+1} F_\bullet))$ is finitely generated. It is known that for a free group $F$ there is a natural isomorphism $\gamma_n F/\gamma_{n+1} F\cong L^n(F_{ab}),$ where $L^n$ is the functor of $n$-th Lie power  (freely generated Lie algebra submodule of degree $n$ homogeneous elements) \cite{MagnusKarrasSolitar}. Therefore $\pi_1(\gamma_n F_\bullet/\gamma_{n+1} F_\bullet)\cong \pi_1(L^n((F_\bullet)_{ab} )).$ Any non-negatively graded chain complex $C_\bullet$ of free abelian groups is quasi-isomorphic (isomorphic in the derived category) to the chain complex of its homology groups $H_*(C_\bullet)$ with trivial differentials. For the chain complex $H_*(C_\bullet)$ with trivial differentials, we can take a resolution for each of its terms, forming a bicomplex with trivial horizontal differentials, take the total complex \cite[Ch. IV, \S 4]{CartanEilenberg} and obtain a new chain complex $D_\bullet$, which is homotopy equivalent to $C_\bullet.$   Then a non-negatively graded  chain complex $C_\bullet$ of free abelian groups with finitely generated $H_0(C_\bullet)$ and $H_1(C_\bullet)$ is homotopy equivalent to a chain complex of free abelian groups $D_\bullet$ with finitely generated $D_0$ and $D_1.$ By the Dold-Kan correspondence it follows that $(F_\bullet)_{ab}$ is homotopy equivalent to a free simplicial abelian group $A_\bullet$ with finitely generated $A_0$ and $A_1.$ Therefore $\pi_1( L^n((F_\bullet)_{ab}))\cong \pi_1( L^n(A_\bullet))$ is finitely generated.
\end{proof}

\begin{Corollary}\label{cor_finpres}
If $G$ is  a group such that $H_1G$ is finitely generated and $H_2G$ is finite, then $\BB_n G$ is finite for any $n.$
\end{Corollary}
\begin{proof}
This follows from Proposition \ref{prop_f.p._baer} and Theorem \ref{th_Ellis}.
\end{proof}

\section{\bf Right exact $\ZZ$-completion of a group.}

In this section we intensively use the theory of Bousfield-Kan completion  \cite{BousfieldKan}. Let $R$ be either a subring of $\mathbb Q$ or $\ZZ/m.$ Then $R_\infty X$ denotes Bousfield-Kan  $R$-completion of a space $X.$ One of the main properties of the Bousfield-Kan $R$-completion $R_\infty X$ is that its homotopy type  is an  invariant of $R$-homological equivalence i.e. if $X\to Y$ is a $R$-homological equivalence, then $R_\infty X\to R_\infty Y $ is a homotopy equivalence. Moreover, there is the following variant of Whitehead's theorem which easily follows from a combination of results of Bousfield-Kan \cite{BousfieldKan} and Kan \cite{Kan76}.

\begin{Theorem}[{see  \cite{Kan76}}]\label{th_whitehead} Let $X\to Y$ be a map of connected pointed spaces such that $H_i(X,R)\to H_i(Y,R)$ is an isomorphism for $i<n$ and an epimorphism for $i=n.$ Then $\pi_i R_\infty X \to \pi_iR_\infty Y$ is an isomorphism for $i<n$ and an epimorphism for $i=n.$
\end{Theorem}
\begin{proof}
Following Bousfield \cite{Bousfield75} and Kan \cite{Kan76} we consider a factorization $X\to Ef \overset{f'}\to Y$ of the map $f:X\to Y$ in which $X\to Ef$ is a cofibration and $R$-homological equivalence $H_*(X,R)\cong H_*(Ef,R)$ and $f'$ is a fibration in the model category where weak equivalences
are homology equivalences (Bousfield localization of a model category).  Kan's version of Whitehead theorem \cite[Theorem 3.1]{Kan76} implies that $\pi_if':\pi_iEf\to \pi_i Y $ is an isomorphism for $i< n$ and an epimorphism for $i=n.$ It follows that \cite[Ch. IV, Proposition 5.1]{BousfieldKan} the map $\pi_i R_\infty f' : \pi_i R_\infty Ef\to \pi_i R_\infty Y$ is an isomorphism for $i< n$ and an epimorphism for $i=n.$ Since $X\to Ef$ is an $R$-homological equivalence, we obtain that $R_\infty X \to R_\infty Ef$ is a homotopy equivalence. The assertion follows.
\end{proof}

For a group $G$ we define its pro-nilpotent completion as the inverse limit of its $n$-nilpotenizations
$$\widehat G=\varprojlim \ \nu_n G,$$
where $\nu_n G=G/\gamma_{n+1} G$ and $\gamma_n G$ is the lower central series.
The functor of pro-nilpotent completion is not right exact. However its zero derived functor is right exact (Remark \ref{rem_l0}).

\begin{Definition}
For a group $G$  we define its {\it right exact $R$-completion} $R_\infty G$ as the fundamental group of the Bousfield-Kan $R$-completion  \cite{{BousfieldKan}} of the classifying space:
$$R_\infty G: =\pi_1(R_\infty BG).$$
The map $BG\to R_\infty BG$ induces a map
$$G\to R_\infty G.$$
In this paper we will be interested only in the case $R=\ZZ.$
\end{Definition}

Further we use the homotopy theory of simplicial groups (\cite[Ch. VI]{May}, \cite[Ch. IV]{BousfieldKan}, \cite{GoerssJardine}). In particular, we use the following pair of adjoint functors:  Kan's loop functor $\mathcal G $ from the category of reduced simplicial sets to the category of simplicial groups and the functor of classifying space of a simplicial group $\overline{\mathcal W}:$
$$ \mathcal G: {\sf sSets}_{\sf red} \rightleftarrows {\sf sGr} : \overline{\mathcal{W}}.$$
We use the following interpretation of the  Bousfield-Kan $\ZZ$-completion of a reduced simplicial set $X_\bullet$
$$\ZZ_\infty X_\bullet \simeq \overline{\mathcal W}( \widehat{\mathcal G X_\bullet})$$
(see \cite[Ch. IV, Prop. 4.1]{BousfieldKan}).

\begin{Proposition}\label{prop_Z-inf_L0} The functor of right exact $\ZZ$-completion of groups is the zero derived functor of the functor of pro-nilpotent completion of groups
$$(G\mapsto \ZZ_\infty G) =  {\sf L}_0( G\mapsto \widehat G).$$
In particular, the functor $G\mapsto \ZZ_\infty G$ is right exact.  Moreover, for any group $G$ there is a natural short exact sequence
\begin{equation}\label{eq_exact_seq_lim1}
1 \longrightarrow
{\varprojlim}^1\:  \BB_n G
\longrightarrow
 \ZZ_\infty G \longrightarrow \widehat{G} \longrightarrow 1.
\end{equation}
\end{Proposition}
\begin{Remark}
The group ${\varprojlim}^1\:  \BB_n G$ is abelian. However, the extension \eqref{eq_exact_seq_lim1} is not necessarily central.
\end{Remark}
\begin{proof}[Proof of Proposition \ref{prop_Z-inf_L0}]
If we take $X_\bullet= \overline{\mathcal W} G,$ where $G$ is considered as a constant simplicial group, we obtain a simplicial classifying space of $G.$ Then $F_\bullet={\mathcal G}( \overline{\mathcal  W} G)\to G$ is a free simplicial resolution of $G,$
where the map $F_\bullet\to G$ is the counit of the adjunction $\mathcal G \dashv \overline{\mathcal{W}}$ \cite[Lemma 5.3]{GoerssJardine}.  Hence
$$\ZZ_\infty G= \pi_1(\ZZ_\infty BG)=\pi_1( \overline{\mathcal W} \widehat F_\bullet )=\pi_0(\widehat F_\bullet).$$
Therefore $G\mapsto \ZZ_\infty G$ is the zero derived functor of the functor of pro-nilpotent completion.  Milnor's short exact sequence has the following form here
$$1\longrightarrow {\varprojlim}^1\:\pi_1(\nu_n F_\bullet) \longrightarrow \pi_0( \varprojlim\  \nu_n F_\bullet ) \longrightarrow \varprojlim\ \pi_0(\nu_n F_\bullet)  \longrightarrow 1.$$
The functor $\nu_n$ is right exact (Theorem \ref{th_right_exact}), and hence, $\pi_0( \nu_n F_\bullet )\cong  \nu_n G.$ By definition $\pi_1(\nu_n F_\bullet)=M_nG$ (Definition \ref{def_Baer}) which gives the required short exact sequence.
\end{proof}

\begin{Proposition} \label{prop_Z-inf-X}
 Let $X$ be a path-connected pointed space. Then
$$\pi_1(\ZZ_\infty X)\cong  \ZZ_\infty (\pi_1(X)).$$
\end{Proposition}
\begin{proof} Consider the classifying space $Y=B(\pi_1X).$
Then the map $X\to Y$ induces an isomorphism $\pi_1X\cong \pi_1Y $ and an epimorphism $\pi_2X\epi \pi_2Y.$ Hence, by \cite[Ch. IV, Proposition 5.1]{BousfieldKan} the map $\ZZ_\infty X \to \ZZ_\infty Y$ induces an isomorphism $\pi_1(\ZZ_\infty X)\cong\pi_1(\ZZ_\infty Y)= \ZZ_\infty (\pi_1 X).$
\end{proof}

\begin{Corollary} The group $\ZZ_\infty (\pi_1 X)$ is an invariant of homological equivalence of spaces.
\end{Corollary}

A group homomorphism $G\to G'$ is said to be {\it $n$-connected} if  the map $H_iG\to H_iG'$  is an isomorphism for $i<n$ and an epimorphism for $i=n.$ Stallings' theorem \cite[Th.3.4]{Stallings}  says that any $2$-connected homomorphism $G\to G'$ induces an isomorphism $\nu_n G\cong \nu_n G'$ for any $n.$ It follows that it induces an isomorphism of pro-nilpotent completions $\widehat G\cong \widehat G'.$

\begin{Proposition}[Stallings' theorem for $\ZZ_\infty G$] \label{prop_Stallings} Let $f:G\to G'$ be a group homomorphism. Then the following holds.
\begin{enumerate}
\item If $f$ is 1-connected, then $\ZZ_\infty G\to \ZZ_\infty G'$ is an epimorphism.
\item If $f$ is 2-connected, then $\ZZ_\infty G\to \ZZ_\infty G' $ is an isomorphism  $\ZZ_\infty G \cong \ZZ_\infty G'.$
\end{enumerate}
\end{Proposition}
\begin{proof}
It follows from Theorem \ref{th_whitehead}.
\end{proof}

\begin{Corollary}
A $2$-connected homomorphism $G\to G'$ induces an isomorphism
$${\varprojlim}^1\ \BB_nG\cong {\varprojlim}^1\ \BB_nG'.$$
\end{Corollary}
\begin{proof}
It follows from Proposition \ref{prop_Stallings} and Proposition \ref{prop_Z-inf_L0}.
\end{proof}

\begin{Proposition}\label{prop_f.p.H2} Let $G$ be a group such that $H_1G$ is finitely generated and $H_2G$ is finite. Then the natural morphism $\ZZ_\infty G \to \widehat G$ is an isomorphism:
$$\ZZ_\infty G \cong \widehat G.$$
\end{Proposition}
\begin{proof}
By Corollary \ref{cor_finpres} the Baer invariants $\BB_n G$ are finite. The Mittag-Leffler condition implies ${\varprojlim}^1\:  \BB_n G=0.$
\end{proof}

If $U$ is a normal subgroup of a group $G,$ then we set $\widehat U_G=\varprojlim\: U/\gamma_n(U,G).$ Note that $\nu_n(G\ltimes U)=\nu_n G \ltimes (U/\gamma_{n+1}(U,G)).$  Then
$$\widehat{G\ltimes U}=\widehat G \ltimes \widehat U_G.$$

\begin{Proposition}\label{prop_resoluion0} Let $U\mono E\epi G$ be a short exact sequence of groups and assume that the map $\ZZ_\infty E\to \widehat E$ is an isomorphism. Then there is an exact sequence
$$ \widehat U_E \longrightarrow \widehat E \longrightarrow \ZZ_\infty G \longrightarrow 1.$$
\end{Proposition}
\begin{proof}
 Since the functor $G\mapsto \ZZ_\infty G$ is right exact, by Theorem \ref{th_right_exact} we have an exact sequence $\ZZ_{\infty,E} U \to \ZZ_\infty E \to \ZZ_\infty G \to 1.$ Using that the map $\ZZ_\infty (E\ltimes U)\to \widehat{E\ltimes U}$ is an epimorphism, we obtain that $\ZZ_{\infty,E} U\to \widehat{U}_E $ is an epimorphism.
Since $\ZZ_\infty E \cong \widehat E,$ the commutative square
$$
\begin{tikzcd}
\ZZ_{\infty,E} U \arrow[r]\arrow[d,twoheadrightarrow] & \ZZ_\infty E\arrow[d,"\cong"] \\
\widehat U_E\arrow[r] & \widehat E
\end{tikzcd}
$$
implies that the map $\ZZ_{\infty,E}U\to \widehat E$ has the same image as $\widehat U_E\to \widehat E.$ Therefore we obtain a short exact sequence $ \widehat U_E \to \widehat E \to \ZZ_\infty G \to 1.$
\end{proof}

\begin{Corollary}\label{prop_resoluion}
Let $U\mono E\epi G$ be a short exact sequence of groups. Assume that  $H_1E$ is finitely generated and $H_2E$ is finite.  Then there is an exact sequence
$$ \widehat U_E \longrightarrow \widehat E \longrightarrow \ZZ_\infty G \longrightarrow 1.$$
\end{Corollary}

\begin{proof}
This follows from Proposition \ref{prop_resoluion0} and Proposition \ref{prop_f.p.H2}.
\end{proof}

\begin{Remark} Corollary  \ref{prop_resoluion} gives an effective way for computing $\ZZ_\infty G.$ It is thus enough
to find an extension $U\mono E\epi G$ with finitely generated $H_1E,$  finite $H_2E$ and computable $\widehat E$ and $\widehat U_E$, and then use the exact sequence.
\end{Remark}

If $G=F/R$ is a free presentation of $G$ we set $\bar R= \varprojlim \: (R\cdot \gamma_n F)/\gamma_n F.$ Note that $\bar  R$ is the closure of $R$ in $\widehat F$ in the limit topology.

\begin{Proposition}\label{prop_construction_by_presentation} If $G=F/R$ is a free presentation of a group $G$, then there is a commutative diagram with exact rows and columns
$$
\begin{tikzcd}
& \widehat R_F\arrow[r]\arrow[d] & \widehat F \arrow[d,equal]\arrow[r] & \ZZ_\infty G \arrow[d,twoheadrightarrow] \arrow[r] & 1\\
1\arrow[r] & \bar R \arrow[d,twoheadrightarrow] \arrow[r] & \widehat F\arrow[r] & \widehat G\arrow[r] & 1\\
& {\varprojlim}^1 \BB_n G. & & &
\end{tikzcd}
$$
\end{Proposition}
\begin{proof} Bousfield and Kan \cite[Ch. IV, \S 1]{BousfieldKan} proved that for any free group $F$  there is a homotopy equivalence $\ZZ_\infty(BF)\cong B\widehat{F}$. Hence the map $\ZZ_\infty F\to \widehat{F}$ is an isomorphism. The first horizontal exact sequence follows from Proposition \ref{prop_resoluion0}. The second horizontal exact sequence follows from the exact sequence $ (R\cdot \gamma_{n+1} F )/\gamma_{n+1} F \mono \nu_n F \epi \nu_n G.$ The exactness of the vertical left hand sequence follows from the snake lemma.
\end{proof}

For a group $G$ we set $\gamma_\omega G =\bigcap_{n=1}^\infty \gamma_n G$ and $\nu_\omega G=G/\gamma_\omega G.$

\begin{Proposition}\label{prop_gamma_omega} If $G$ is a group with finitely generated $H_1G,$ then the maps $G\to \ZZ_\infty G \to \widehat G$ induce isomorphisms
$$ \nu_n G \cong \nu_n (\ZZ_\infty G) \cong \nu_n \widehat G$$
for any $n,$
and the maps from Proposition \ref{prop_Z-inf_L0} define isomorphisms
$$\gamma_\omega( \ZZ_\infty G )\cong {\varprojlim}^1\: \BB_nG, \hspace{1cm} \nu_\omega( \ZZ_\infty G) \cong \widehat G.$$
\end{Proposition}
\begin{proof}
In this proof we use the theory of $H\ZZ$-localizations of groups \cite{Bousfield75},  \cite{Bousfield77}. For a group $G$ we denote its $H\ZZ$-localization by $LG.$ We use the fact that the homomorphism to $H\ZZ$-localization $G\to LG$ is 2-connected (by definition, see \cite[Def.5.1]{Bousfield75} or \cite[page 1]{Bousfield77}) and the fact that localization takes epimorphisms to epimorphisms \cite[Cor.2.13]{Bousfield77}. Moreover, if $G\to G'$ is a 1-connected homomorphism, then $LG\to LG'$ is an epimorphism \cite[Cor. 2.13]{Bousfield77}.
In particular, Stalling's theorem implies that $\nu_n G\cong \nu_n LG.$

Let $G$ be a group with finitely generated $H_1G.$ Consider a 1-connected homomorphism $F\to G$ from a finitely generated free group. Then the map $LF\epi LG$ is an epimorphism and $H_2LF=0$ because $F\to LF$ is 2-connected. Set $U={\sf Ker}(LF\epi LG).$ Then by Corollary  \ref{prop_resoluion}  there is an exact sequence
$$\widehat U_{LF} \longrightarrow \widehat{LF} \longrightarrow \ZZ_\infty L G \longrightarrow 1.$$
Using the Stallings theorem for $\ZZ_\infty$ (Proposition \ref{prop_Stallings}) and the classical Stallings theorem, we obtain that $\ZZ_\infty LG \cong \ZZ_\infty G$ and $\widehat{LF}=\widehat F.$ Hence, we have an  exact sequence
$\widehat U_{LF} \to \widehat{LF} \epi \ZZ_\infty G.$
The exact sequence $U/(U\cap \gamma_{n+1}LF) \mono \nu_n L F\epi \nu_n G$ induces a short exact sequence $\bar U \mono \widehat{LF} \epi \widehat G,$ where $\bar U=\varprojlim U/(U\cap \gamma_{n+1}LF).$ It is known  \cite[Th. 13.3]{Bousfield77} (see also \cite{BaumslagStammbach}) that  $\nu_n F\cong \nu_n \widehat F$ and $\nu_n G \cong \nu_n \widehat G$ (this is why $F$ was chosen to be finitely generated).
The map $\widehat U_{LF}\to \nu_nLF$ equals to the composition $\widehat U_{LF}\epi  U/\gamma_{n+1}(U,LF) \to \nu_n LF.$ Hence, the image of $\widehat U_{LF}$ in $\nu_n LF$ equals to the image of $U$ in $ \nu_n LF.$
The map $\bar U\to \nu_n LF$ equals to the composition $\bar U\epi U/(U\cap \gamma_{n+1} LF) \to \nu_n LF.$ Hence, the image of $\bar U$ in $\nu_n LF$ also equals to the image of $U.$  Then the exact sequences $\widehat U_{LF} \to \nu_nLF \to \nu_n(\ZZ_\infty G) \to 1 $ and $ \bar U \to \nu_nLF \to \nu_n\widehat G \to 1$ imply that the map $\ZZ_\infty G\epi \widehat G $ gives rise to an isomorphism $ \nu_n(\ZZ_\infty G)\cong \nu_n \widehat G.$

Set $K={\sf Ker}(\ZZ_\infty G \epi \widehat G ).$ By Proposition \ref{prop_Z-inf_L0}  we have $K\cong {\varprojlim}^1\: \BB_n G.$ Since $\gamma_\omega(\widehat G)=1,$ we have $\gamma_\omega(\ZZ_\infty G )\subseteq K.$ On the other hand the isomorphism $\nu_n \ZZ_\infty G \cong \nu_n \widehat G$ implies that $K\subseteq \gamma_{n+1} \ZZ_\infty G$ holds for any $n.$ Thus $K=\gamma_\omega(\ZZ_\infty G).$
\end{proof}

\begin{Proposition}\label{prop_nilpotent_r}
For a nilpotent group $N$ the natural map $N\to \ZZ_\infty N$ is an isomorphism
$$N\cong  \ZZ_\infty N.$$
\end{Proposition}
\begin{proof}
Consider a free presentation $N=F/R.$ Then $\gamma_k F\subseteq  R  $ for some $k.$ It follows that $\gamma_{n+k} F  \subseteq \gamma_n(R,F)$ for any $n.$ Therefore the map $\BB_{n+k} N\to \BB_{n}N$ is trivial for any $n$. Hence ${\varprojlim}^1\: \BB_n N=0.$
\end{proof}

\begin{Remark} The Proposition \ref{prop_nilpotent_r} can be also easily deduced from the theory of Bousfield-Kan of nilpotent spaces \cite[Ch.V.3]{BousfieldKan}.
\end{Remark}

\section{\bf Examples of 3-manifolds with isomorphic $\widehat \pi_1$ but non-isomorphic $ \ZZ_\infty \pi_1$.}

For a homeomorphism of a manifold $f:X\to X$ one can consider its mapping torus:
$$X_f= \frac{[0,1]\times X}{(1,x)\sim (0,f(x))}.$$
Then $X_f$ is also a manifold such that ${\sf dim}(X_f)={\sf dim}(X)+1.$ Moreover, there is a fiber bundle
$$X \to X_f \epi S^1 $$
with the monodromy action induced by $f.$

Consider the following matrices
$$a_k=
\left(\begin{matrix}
-1 & k \\
0 & -1
\end{matrix}
\right),
$$
where $k$ is an odd integer. They define homeomorphisms $\mathbb R^2\to \mathbb R^2,$ and thus, on the quotients $f_k:(S^1)^2\to (S^1)^2.$ Their mapping tori are 3-manifolds which will be denoted by $X_k.$ The homotopy long exact sequence of the fibration $(S^1)^2\to X_k\epi S^1$ implies that
$$X_k \simeq K(G_k,1),$$
where $G_k = \ZZ \ltimes_{a_k} \ZZ^2.$

\vbox{
\begin{Theorem}\label{th_manifolds} Let $k$ and $l$ be two odd integers.
\begin{enumerate}
\item Then $
\widehat{ \pi_1(X_k)} \cong \widehat{ \pi_1(X_l)}.$
\item If $kl \not\equiv 1,7\ ({\rm mod}\ 8),$ then $  \ZZ_\infty  \pi_1(X_k)\not \cong  \ZZ_\infty \pi_1(X_l).$
\end{enumerate}
\end{Theorem}
}

\

In order to prove this theorem, we need to prove several lemmas.
Further we will always assume that $k,l$ are odd integers.

\begin{Lemma}\label{lemma_G_k_compl} There is an isomorphism
$\widehat G_k\cong  \ZZ\ltimes_{a_k} \ZZ^2_2,$
where $\ZZ_2$ is the group of $2$-adic integers. Moreover, the group $\ZZ\ltimes_{a_k} (\ZZ/2^n)^2$ is nilpotent for any $n.$
\end{Lemma}
\begin{proof}
It is easy to check that for arbitrary matrix $a\in {\rm GL}_s(\ZZ)$ we have $\gamma_{n+1}(\ZZ\ltimes_a \ZZ^s)= 0\ltimes b^n(\ZZ^s),$ where $b=a-1$ and $n\geq 1.$ Set $b_k=a_k-1.$ Then $\gamma_{n+1}(G_k)=0\ltimes b_k^n(\ZZ^2).$  By induction we prove that
$$b_k^n= \left(
\begin{matrix}
(-2)^n & (-2)^{n-1} nk \\
0 & (-2)^n
\end{matrix}\right).
 $$
 Hence $b_k^n(\ZZ^2) \subseteq  2^{n-1}\ZZ^2.$  On the other hand
 $$b_k^n \cdot  \left(
\begin{matrix}
(-2)^n & -(-2)^{n-1} nk \\
0 & (-2)^n
\end{matrix}\right) = \left(
\begin{matrix}
(-2)^{n^2} &  0 \\
0 & (-2)^{n^2}
\end{matrix}\right). $$
This implies that $2^{n^2}\ZZ^2 \subseteq b_k^n(\ZZ^2).$ Therefore the filtrations $b^n_k(\ZZ^2)$ and $2^n\ZZ^2$ of $\ZZ^2$ are equivalent.  It follows that $\widehat G_k=\ZZ\ltimes_{a_k} \ZZ^2_2$ and $\ZZ\ltimes_{a_k} (\ZZ/2^n)^2$ is nilpotent for any $n.$
\end{proof}

\begin{Lemma}\label{lemma_G_k_iso} There  is an isomorphism $\widehat G_k\cong \widehat G_l.$
\end{Lemma}
\begin{proof}
Generally, if we have two groups $\ZZ\ltimes_{f_1} H_1$ and $\ZZ\ltimes_{f_2} H_2$ where $f_i\in {\sf Aut}(H_i)$, then a homomorphism $\varphi:H_1\to H_2$ satisfying $\varphi f_1=f_2 \varphi$ defines a homomorphism $1\ltimes \varphi: \ZZ\ltimes_{f_1} H_1\to \ZZ\ltimes_{f_2} H_2.$  We take $\varphi = \left(
\begin{matrix}
l &  0 \\
0 & k
\end{matrix}\right)$ as a matrix in ${\rm GL}_2(\ZZ_2).$ Here we use that odd numbers are invertible in $\ZZ_2.$ A direct computation shows that
$a_l \varphi = \varphi a_k.$
Therefore $\varphi$ defines an isomorphism $1\ltimes \varphi : \ZZ\ltimes_{a_k} \ZZ_2^2 \to \ZZ\ltimes_{a_l} \ZZ_2^2. $ Thus $\widehat G_k\cong \widehat G_l.$
\end{proof}

Let $F(x,y)$ be the free group generated by $x$ and $y.$ We denote by $N$ the free 2-generated nilpotent group of class $2$. In other words $N=\nu_2 F(x,y).$ The images of $x$ and $y$ in $N$ will be denoted by the same letter. It has the following presentation
$$N=\langle x,y \mid [x,y,y]=[x,y,x]=1 \rangle,
$$
where the triple commutator is defined as $[a,b,c]=[[a,b],c].$
There exists a unique automorphism of $F(x,y)$ such that $x\mapsto x^{-1}$ and $y\mapsto x^{k}y^{-1}.$ It induces an automorphism
$$a'_k:N\to N, \hspace{1cm} a'_k(x)=x^{-1}, \ a_k'(y)=x^{k}y^{-1}.$$ It lifts the automorphism $a_k:\ZZ^2\to \ZZ^2.$ Note that $a_k'([x,y])=[x,y].$
Define the following group
$$E_k = \ZZ \ltimes_{a'_k} N.$$
Then $(0,[x,y]) $ is in the center of this group. Therefore we have a central extension
\begin{equation}\label{eq_Ek_central_ext}
 1\longrightarrow \ZZ \longrightarrow E_k \longrightarrow G_k \longrightarrow 1.
\end{equation}
It is easy to check that $E_k$ has a following presentation

\begin{equation} \label{eq_pres_e_k}
E_k=\langle a, x,y \mid x^a=x^{-1},\  y^a=x^{k}y^{-1},\ [x,y,y]=[x,y,x]=1  \rangle.
\end{equation}

We will denote the images of $x,y$ in $\widehat  E_k$ by the same letters.

\begin{Lemma}\label{lemma_E_k}
There is an isomorphism $H_2 E_k\cong \ZZ/4,$ an isomorphism $\ZZ_\infty E_k  \cong  \widehat E_k$ and an exact sequence
$$ \ZZ \overset{[x,y]}\longrightarrow \widehat E_k \longrightarrow  \ZZ_\infty G_k \longrightarrow 1.$$
\end{Lemma}
\begin{Remark}
The group $E_k$ is ``better'' than the group $G_k$ because its second homology group is finite, and hence, its right exact completion equals to its usual completion (Proposition \ref{prop_f.p.H2}). We think about $E_k$ as about a resolution of $G_k.$ Recall that, the same type of construction is used in \cite{IM}, in order to prove that the $H{\mathbb Z}$-length of a free noncyclic group is $\geq \omega+2$.
\end{Remark}
\begin{proof}[Proof of Lemma \ref{lemma_E_k}]
 The spectral sequence of the extension $N\mono  E_k \epi \ZZ$ gives a short exact sequence
$$ 1\longrightarrow (H_2 N)_\ZZ \longrightarrow  H_2 E_k \longrightarrow  (H_1 N)^\ZZ \longrightarrow 1.$$
There is an isomorphism $H_1N=\ZZ^2,$ where $\ZZ$ acts on $\ZZ^2$ by $a_k.$ Hence $(H_1N)^\ZZ=0$ and $H_2E_k=(H_2N)_\ZZ.$
Hopf's formula gives an isomorphism $H_2N=\gamma_3 F /\gamma_4 F ,$ where $F=F(x,y)$ is the free group. The two elements $[x,y,y],[x,y,x]$ form a basis of $H_2 N\cong \ZZ^2.$ Hopf's formula is natural in the following sense: if we have two presentations of two groups $G=F/R$ and $G'=F'/R'$ and a homomorphism $\varphi: G\to G'$ then any lifting of the homomorphism to the free groups $F\to F' $ induces the  homomorphism $(R\cap [F,F])/ [R,F] \to (R'\cap [F',F'])/ [R',F']$ corresponding to $\varphi_*:H_2 G\to H_2 G'.$    Then the action of $\ZZ$ on $H_2 N$ is the following: $[x,y,x ] \mapsto [x^{-1},x^{k}y^{-1},x^{-1}],$  and
$ [x,y,y] \mapsto [x^{-1},x^{k} y^{-1},x^{k}y^{-1}].$ Let us treat the elements $[x^{-1},x^{k}y^{-1},x^{-1}]$ and $[x^{-1},x^{k} y^{-1},x^{k}y^{-1}]$ as elements of the Lie ring ${\sf gr} (F)=\bigoplus \gamma_n F /\gamma_{n+1}F$ with the additive notation.
Then, using the fact that  ${\sf gr} (F)$ is a Lie ring, where sum is induced by the product and the bracket is induced by the commutator, we obtain
\begin{equation*}
\begin{split}
[x^{-1},x^{k}y^{-1},x^{-1}]&=[-x,kx-y,-x]=[-x,kx,-x]+[-x,-y,-x]=-[x,y,x],\\
[x^{-1},x^{k} y^{-1},x^{k}y^{-1}]&=[-x,kx-y,kx-y]=[-[x,kx]+[x,y],kx-y]=\\
&=[x,y,kx-y]=k[x,y,x] -[x,y,y].
\end{split}
\end{equation*}
Then the action of $\ZZ$ on $H_2N$ is the following
\begin{equation}
\begin{split}
[x,y,x] & \mapsto\ \: -[x,y,x], \\
[x,y,y] & \mapsto k\cdot [x,y,x] - [x,y,y].
\end{split}
\end{equation}
It follows that $\ZZ$ acts on $H_2N$ by the matrix $a_k.$ Then $(H_2N)_\ZZ$ is isomorphic to the quotient $\ZZ^2/b_k(\ZZ^2),$ where $b_k=a_k-1.$  Computing the Smith normal form of the matrix $b_k=\left( \begin{smallmatrix}
-2 & k \\
0 & -2
\end{smallmatrix} \right) $
we obtain the matrix
$ \left( \begin{smallmatrix}
4 & 0 \\
0 & 1
\end{smallmatrix} \right)  .$
Therefore $H_2E_k=(H_2N)_\ZZ=\ZZ/4.$ Proposition \ref{prop_f.p.H2} implies that $\ZZ_\infty E_k \cong \widehat E_k.$
Then the central extension \eqref{eq_Ek_central_ext} together with Corollary \ref{prop_resoluion} give the exact sequence $\ZZ\to \widehat E_k \to \ZZ_\infty G_k \to 1. $
 \end{proof}

Note that $[N,N]$ is the cyclic group generated by $[x,y].$ Hence,
 there is a central extension
\begin{equation}\label{eq_N_short_exact}
1\longrightarrow \ZZ \longrightarrow N \longrightarrow  \ZZ^2 \longrightarrow 1.
\end{equation}
Any element of $N$ can be uniquely presented as $x^s y^t [x,y]^u.$ Note that for any $s,t\in \ZZ$ the following holds in $N$
$$[x^s,y^t]=[x,y]^{st}.$$ 
The product in $N$ can be defined by the formula
\begin{equation} \label{eq_prod_formula}
(x^s y^t [x,y]^u)(x^{s'} y^{t'} [x,y]^{u'})=x^{s+s'} y^{t+t'} [x,y]^{u+u'-ts'}.
\end{equation}
This means that the extension \eqref{eq_N_short_exact} can be defined by the $2$-cocycle $\alpha:\ZZ^2 \times \ZZ^2 \to \ZZ$ given by
$$\alpha( ( s,t) , (s',t') )=-ts'.$$

Let $R$ be a ring.
We denote by $N\otimes R$ the group of formal expressions $x^s y^t [x,y]^u$ where  $s,t,u\in R,$ and the multiplication is defined by the formula   \eqref{eq_prod_formula}. This group can be obtained as the group corresponding  to the $2$-cocycle $ \alpha_R:R^2\times R^2\to R$ given by the same formula $\alpha_R((s,t),(s',t'))=-ts'$
$$1 \longrightarrow R \longrightarrow N\otimes R \longrightarrow R^2 \longrightarrow 1.$$
In order to prove that this construction gives a well-defined group $N\otimes R$ one just needs  to check that $ \alpha_R$ is a $2$-cocycle, which can be done by a direct computation.

This construction is natural by $R.$ In particular, the ring homomorphism $\ZZ\to R$ gives a group homomorphism
$N\to N\otimes R.$

Let's describe the action of $a_k'$ on $N$ more explicitly. It is easy to check that
$$a_k'(y^t)=(x^ky^{-1})^t= x^{kt} y^{-t} [x,y]^{kt(t-1)/2 }$$
and $a_k'([x,y])=[x,y].$ Here we use the formula that holds in $N$ 
$$[x^{-1},y^{-1}]=[x,y]^{(-1)(-1)}=[x,y].$$
It follows that
$$a_k'(x^sy^t[x,y]^u) = x^{-s+kt} y^{-t} [x,y]^{u+kt(t-1)/2}.$$
It is easy to see that the same formula defines an automorphism of $N\otimes \ZZ_2:$
$$a_{k,\ZZ_2}' :N\otimes \ZZ_2 \longrightarrow N\otimes \ZZ_2$$
because elements of the form $t(t-1)$ are uniquely divisible by $2$ in $\ZZ_2.$
For simplicity we will denote  $a'_k=a'_{k,\ZZ_2}.$

\begin{Lemma} The obvious map $E_k \to \ZZ\ltimes_{a_k'} (N\otimes \ZZ_2)$ induces an isomorphism
$$ \widehat E_k \cong \ZZ\ltimes_{a'_k} (N\otimes \ZZ_2)  $$
and there are short exact sequences
$$
\begin{tikzcd}
0 \arrow[r] & \ZZ\arrow[r,"{[x,y]}"] \arrow[d] & \widehat E_k \arrow[r] \arrow[d,equal] & \ZZ_\infty G_k \arrow[r] \arrow[d] & 1
\\
0 \arrow[r] & \ZZ_2 \arrow[r] & \widehat E_k \arrow[r] & \widehat G_k \arrow[r] & 1
\end{tikzcd}
$$
 Moreover, the group $\ZZ\ltimes_{a_k'} (N\otimes \ZZ/2^n)$ is nilpotent for any $n.$
\end{Lemma}
\begin{proof} Set $\mathcal E=E_k.$
Consider a subgroup of
$\mathcal E_n \subseteq N $
$$\mathcal E_n = \{ x^sy^t[x,y]^u \mid s,t \in 2^{n+1}\ZZ, u\in 2^{n}\ZZ\}.$$
Since $t\in 2^{n+1}\ZZ$ implies that $t(t-1)/2 \in 2^n\ZZ,$ the explicit formula for $a'_k$ gives that
$$a'_k( \mathcal E_n ) \subseteq \mathcal E_n.$$ Therefore $\mathcal E_n=0\ltimes \mathcal E_n$ is a normal subgroup of $\mathcal E.$
It is easy to check that $N\otimes \ZZ_2\cong \varprojlim N/\mathcal E_n.$ Then $ \ZZ\ltimes_{a_k'} (N\otimes \ZZ_2) \cong \varprojlim\:  \mathcal E/ \mathcal E_n.$
We also set $\mathcal G=G_k$ and $\mathcal G_n= \ZZ\ltimes_{a_k} 2^{n+1} \ZZ^2.$
Note that there is a central extension
$$ 1\longrightarrow \ZZ/2^n \longrightarrow \mathcal E/\mathcal  E_n  \longrightarrow \mathcal G/\mathcal G_{n} \longrightarrow 1. $$
By Lemma \ref{lemma_G_k_compl} $\mathcal G/\mathcal G_n $ is nilpotent, and hence, $\mathcal E / \mathcal E_n$ is nilpotent.  It  follows that for any $n$ there exists $m(n)$ such that $\gamma_{m(n)} \mathcal E \subseteq  \mathcal E_n.$ Hence we have the following diagram
$$
\begin{tikzcd}
1\arrow[r]
&
{[N,N]}/({[N,N]}\cap \gamma_{m(n)} \mathcal E)\arrow[d] \arrow[r]
 &
  \mathcal E/\gamma_{m(n)} \mathcal E \arrow[r]  \arrow[d]
  &
  \mathcal G/\gamma_{m(n)} \mathcal  G \arrow[r] \arrow[d]
  &
  1
 \\
1\arrow[r] & \ZZ/2^n \arrow[r] & \mathcal E/\mathcal E_n \arrow[r] & \mathcal G/\mathcal G_n \arrow[r] & 1
\end{tikzcd}
 $$
We want to prove that the central map induce an isomorphism on limits. By Lemma \ref{lemma_G_k_compl} the right hand vertical map induces an isomorphism on limits. Since the towers consist of epimorphisms we only need to prove that the left hand map induces an isomorphism on limits (\cite[Ch.IX, Prop. 2.3, Prop. 2.4]{BousfieldKan}). It is enough to prove that the filtrations $[N,N]\cap \gamma_n\mathcal E $ and  $\langle [x,y]^{2^n} \rangle $ of the cyclic group $[N,N]=\langle [x,y] \rangle$ are equivalent. We already know that $N\cap \gamma_{m(n)} \mathcal E \subseteq \langle [x,y]^{2^n} \rangle  $ because the left hand vertical map is well defined. On the other hand, if we use the presentation   $\mathcal E= \langle x,y,a\mid x^a=x^{-1}, y^a=x^ky^{-1}, [x,y,y]=[x,y,x]=1   \rangle$ (see \eqref{eq_pres_e_k}) we obtain
$$[x, \underbrace{ a,\dots,a}_n,y]=[x,y]^{(-2)^n} \in [N,N] \cap \gamma_{n+2} \mathcal E. $$
Hence $\langle [x,y]^{2^n} \rangle\subseteq [N,N] \cap \gamma_{n+2} \mathcal E.$
It follows that $\widehat{\mathcal E}= \ZZ\ltimes_{a_k'} (N\otimes \ZZ_2).$

The short exact sequences follow from the description of $\widehat E_k$ and Lemma \ref{lemma_E_k}.
\end{proof}

Denote by $\mathcal N$ the quotient of $N\otimes \ZZ_2$ by the subgroup $[N,N]=\{[x,y]^n\mid n\in \ZZ\}.$ Then any element of the group $\mathcal N$ has the form $x^sy^y [x,y]^u,$ where $s,t\in \ZZ_2$ and $u\in \ZZ_2/\ZZ:$
$$\mathcal N= (N\otimes \ZZ_2)/\langle [x,y] \rangle  ,\hspace{1cm} \mathcal N=\{ x^sy^y [x,y]^u\mid  s,t\in \ZZ_2, u\in \ZZ_2/\ZZ\}.$$ The product of elements in $\mathcal N$ is given by
$$x^{s}y^{t} [x,y]^{u} \ \cdot \ x^{s'}y^{t'} [x,y]^{u'}=x^{s+s'} y^{t+t'} [x,y]^{u+u'-\overline{ts'}}, $$
where $\overline{ts'}$ is the image of $ts'$ in $\ZZ_2/\ZZ.$ In particular, for all $s,t\in \ZZ_2$ we have
$$[x^s,y^t]=[x,y]^{ \overline{st} }.$$

\begin{Corollary}\label{cor_4.7} There is an isomorphism $$\ZZ_\infty G_k \cong \ZZ\ltimes_{a'_k} \mathcal N.$$
Moreover
$$\gamma_\omega(\ZZ_\infty G_k)\cong \ZZ_2/\ZZ, \hspace{1cm} \nu_\omega ( \ZZ_\infty G_k) \cong \widehat G_k.$$
\end{Corollary}

\begin{Corollary} There is an isomorphism
$${\varprojlim}^1 \BB_nG_k \cong \ZZ_2/\ZZ.$$
\end{Corollary}

\subsection*{Proof of Theorem \ref{th_manifolds}}
We proved that $\widehat G_k \cong \widehat G_l$ in Lemma \ref{lemma_G_k_iso}.

 Then we only need to prove that $kl\not\equiv 1,7 ({\rm mod}\ 8)$ implies  $\ZZ_\infty G_k \not\cong \ZZ_\infty G_l.$   Assume the contrary, that there is an isomorphsm $ \varphi: \ZZ_\infty G_k \to \ZZ_\infty G_l.$ Then it induces an isomorphism on quotients by $\gamma_\omega$ which are equal to the usual completions $\varphi' : \widehat G_k \to \widehat G_l$ (see Corolary \ref{cor_4.7}) Since $\widehat G_k\cong \ZZ\ltimes \ZZ_2^2,$ we have that the set of $3$-divisible elements of $\widehat G_k$ is $ 0\ltimes \ZZ_2^2.$ Therefore we obtain an isomorphism   $\varphi'':\ZZ_2^2 \to \ZZ_2^2$ and an isomorphism $\tilde \varphi: \ZZ\to \ZZ:$
$$
\begin{tikzcd}
 1\arrow[r]& \ZZ_2^2  \arrow[r]\arrow[d,"\varphi''"] & \widehat G_k \arrow[r]\arrow[d,"\varphi'"] & \ZZ \arrow[r]\arrow[d," \tilde \varphi"] & 1 \\
1\arrow[r] & \ZZ_2^2 \arrow[r] & \widehat G_l \arrow[r] & \ZZ \arrow[r] & 1.
\end{tikzcd}
$$
Note that $\tilde \varphi(1)=\pm 1.$  Therefore for $a=(1,0)\in \widehat{G}_k$ we have
$$ \varphi'(a)=a^{\pm 1} x^{s_0} y^{t_0}$$
for some $s_0,t_0\in \ZZ_2.$
  This implies that $$ \varphi'' a_k  =a_l^{\pm 1} \varphi''.$$
 Any homomorphism $\ZZ_2\to \ZZ_2$ is the multiplication by a 2-adic number. Therefore $\varphi''$ is given by an invertible $2\times 2$-matrix over $\ZZ_2:$
$$ \varphi''= \left(\begin{matrix}
 \varphi_{11} & \varphi_{12}\\
 \varphi_{21} & \varphi_{22}
\end{matrix} \right), \hspace{1cm}  \varphi_{ij}\in \ZZ_2 .$$
The equation $ \varphi'' a_k=a_l^{\pm 1} \varphi''$ implies
$$\left(\begin{matrix}
 -\varphi_{11} &  -\varphi_{12}+ k\varphi_{11}\\
 -\varphi_{21} &   -\varphi_{22} + k\varphi_{21}
\end{matrix} \right)=\varphi '' a_k=a_l^{\pm} \varphi'' = \left(\begin{matrix}
 -\varphi_{11}\pm l\varphi_{21} &   -\varphi_{12}\pm l\varphi_{22}\\
 -\varphi_{21} &   -\varphi_{22}
\end{matrix} \right).$$
Since $k,l\neq 0,$ this implies that $\varphi_{21}=0$ and $k\varphi_{11}= \pm l\varphi_{22}.$ Therefore
$$ \varphi''= \left(\begin{matrix}
 \alpha  l& \beta \\
 0 &\alpha  k
\end{matrix} \right),$$
where $\alpha=\varphi_{22}/k=\pm \varphi_{11}/l $ and $\beta=\varphi_{12}.$ Here we use that odd numbers are invertible in $\ZZ_2.$ Since $\varphi''$ is invertible, $\alpha$ is also invertible in $\ZZ_2.$ If we present elements of $\ZZ_\infty G_k=\ZZ\ltimes \mathcal N $ in the form $a^nx^sy^t[x,y]^u,$ where $n\in \mathbb Z, $ $s,t\in \ZZ_2$ and $u\in \ZZ_2/\ZZ,$ then this means that $$ \varphi(x^s)=x^{\alpha ls} [x,y]^{u_0(s)}, \hspace{1cm} \varphi(y^t)=x^{\beta t} y^{\alpha k t} [x,y]^{v_0(t)}$$
for some $u_0(s),v_0(t)\in \ZZ_2/\ZZ.$ For any $s\in\ZZ_2 $ we have
$$[x^s,y]=[x,y]^{\bar s}.$$
Therefore
$$ \varphi([x,y]^{\bar s})=\varphi([x^{s},y])=[\varphi(x^{s}),\varphi(y) ]=[x^{\alpha l s },x^\beta  y^{\alpha k } ]=[x^{\alpha l s }, y^{\alpha k } ]=[x,y]^{ \overline{\alpha^2s} kl }.$$
Since $\varphi$ is well defined, $s\in \ZZ$ implies $\alpha^2 kls\in \ZZ.$ In particular, $\alpha^2 kl\in \ZZ.$ We set
$$m:=\alpha^2kl\in \ZZ.$$
Since $\alpha,k,l$ are invertible in $\ZZ_2,$ $m$ is odd.
Then $[x,y]^{1/m}$ is a well-defined element of $\mathcal N$ and
$$\varphi([x,y]^{1/m})=[x,y]=1.$$
Since $\varphi$ is an isomorphism, we obtain $[x,y]^{1/m}=1.$
 Therefore $1/m \in \ZZ$ and hence $m=\pm 1.$ Thus
 $$\pm kl=\alpha^{2}.$$
The element $1$  is the only  square in $\ZZ/8.$ Hence $\pm kl \equiv 1 ({\rm mod}\ 8).$ This contradicts our assumption. $\square$


\begin{thebibliography}{99}\bibitem{AkhtiamovIvanovPavutnitskiy}
D. Akhtiamov, S. O. Ivanov, F. Pavutnitskiy. Right exact localizations of groups. 	{\tt arXiv:1905.07612}.

\bibitem{BaumslagStammbach} G. Baumslag, U. Stammbach. On the inverse limit of free nilpotent groups. Commentarii Mathematici Helvetici, 52(1), 219–233 (1977).


\bibitem{Bousfield75} A. K. Bousfield, The localization of spaces with respect to homology, Topology 14 (1975), 133–150.

\bibitem{Bousfield77} A. K. Bousfield. Homological localization  towers  for groups and  $\pi$-modules, Mem. Amer. Math. Soc., vol. 10, no. 186, 1977.

\bibitem{BousfieldKan} A. K. Bousfield,  D. Kan. Homotopy limits, completions and localizations, Lecture Notes in
Mathematics 304, (1972).



\bibitem{BurnsEllis97} J. Burns, G. Ellis. On the nilpotent multipliers of a group. Math. Z. 226(1997),
No. 3, 405–428.

\bibitem{CartanEilenberg} H. Cartan, S. Eilenberg. Homological algebra. Princeton Landmarks in Mathematics. Princeton University Press, Princeton, NJ, 1999.


\bibitem{ChaOrr2013} J.C. Cha and K.E. Orr: Hidden Torsion, 3-Manifolds, and Homology Cobordism, J. Topol. 6 (2013), 490–512.

\bibitem{ChaOrrPrep} J.C. Cha and K.E. Orr: Transfinite Milnor Invariants for 3-Manifolds, 	{\tt arXiv:2002.03208}.

\bibitem{Dold} A. Dold. Homology of symmetric products and other functors of complexes. Ann. Math., 68:54–80, 1958.

\bibitem{DwyerKan} W. G. Dwyer, D. M. Kan. Homotopy theory and simplicial groupoids, Nederl. Akad. Wetensch. Indag. Math. 46 (1984), no. 4, 379–385


\bibitem{Ellis} G. Ellis. A Magnus-Witt type isomorphism for non-free groups, Georgian Math. J. 9 (2002), 703–708.

\bibitem{GoerssJardine} P. G. Goerss, J. F. Jardine, Simplicial homotopy theory. Progress in Mathematics, 174. Birkh¨auser Verlag, Basel, 1999

\bibitem{IM} S.O. Ivanov, R. Mikhailov. On lengths of HZ-localization towers, Isr. J. Math., 226 (2018), 635-683.

\bibitem{Kan} D. Kan, On homotopy theory and c.s.s. groups, Ann. of Math. 68 (1958),
38–53.

\bibitem{Kan76} D. M. Kan. A Whitehead theorem. Algebra, Topology, and Category Theory, edited by A. Heller and M. Tierney, Academic Press, New York, 1976.

\bibitem{Keune} F. Keune. The relativization of $K_2$, J. Algebra 54 (1) (1978), 159–177.

\bibitem{MacDonald} J. L. MacDonald. Group derived functors. J. Algebra 10(1968), 448–477.

\bibitem{MagnusKarrasSolitar} W. Magnus, A. Karras, D. Solitar. Presentations of Groups in Terms of Generators and
Relations, Dover Publications, (2004).

\bibitem{May} P. May. Simplicial objects in algebraic topology. Van Nostrand Mathematical Studies, No. 11. D.
Van Nostrand Co., Inc., Princeton, N.J.-Toronto, Ont.-London, 1967.

\bibitem{MikhailovPassiBook} R. Mikhailov, I. B. S. Passi, Lower central and dimension series of groups,
Lecture Notes in Mathematics, Springer-Verlag, Berlin, 2009.

\bibitem{Milnor} J. Milnor. Isotopy of links, Algebraic Geometry and Topology, Princeton Univ. Press. Princeton,
NJ, (1957).



\bibitem{Quillen} D. Quillen. Homotopical algebra. Lecture Notes in Mathematics, Vol. 43, Springer-Verlag, 1967.

\bibitem{Stallings}  J. Stallings. Homology and central series of groups, J. Algebra 2 (1965), 170–181.

\end{thebibliography}
\end{document}